\documentclass[12pt]{amsart}
\usepackage{amsmath,amssymb,amsfonts,enumerate,amsthm,amscd}
\usepackage[all]{xy}

\numberwithin{equation}{section}

\newtheorem{thm}{Theorem}[section]
\newtheorem{lem}[thm]{Lemma}
\newtheorem{ques}[thm]{Question}
\newtheorem{cor}[thm]{Corollary}
\newtheorem{prop}[thm]{Proposition}
\newtheorem{rem}[thm]{Remark}
\newtheorem{defin}[thm]{Definition}
\newtheorem{defen rem}[thm]{Definition and Remark}
\newtheorem{exam}[thm]{Example}

\newtheorem{defen}[thm]{Definition}

\newcommand{\fa}{\mathfrak{a}}
\newcommand{\fb}{\mathfrak{b}}

\newcommand{\m}{\mathfrak{m}}
\newcommand{\N}{\mathbb{N}}
\newcommand{\Z}{\mathbb{Z}}
\newcommand{\depth}{\operatorname{depth}}
\newcommand{\pd}{\operatorname{pd}}
\newcommand{\reg}{\operatorname{reg}}
\newcommand{\Fdepth}{\operatorname{F-depth_L}}

\newcommand{\rad}{\operatorname{rad}}
\newcommand{\en}{\operatorname{end}}
\newcommand{\ara}{\operatorname{ara}}

\newcommand{\cd}{\operatorname{cd}}
\newcommand{\grade}{\operatorname{grade}}

\begin{document}
\bibliographystyle{amsplain}

\address{ }
\email{}
\author{Maryam Jahangiri}
\address{Faculty of Mathematical Sciences and Computer, Kharazmi University, Tehran,
Iran; and School of Mathematics, Institute for Research in Fundamental Sciences (IPM), P.O.
Box: 19395-5746, Tehran, Iran.}
\email{mjahangiri@ipm.ir, jahangiri.maryam@gmail.com }

\address{}
\email{}

\thanks{\today }
\subjclass[2000]{ 13D45, 13E10.
This research was in part supported by a grant from IPM (No. 90130111)}

\title[Castelnuovo-mumford regularity and cohomological dimension  ]
{Castelnuovo-mumford regularity and cohomological dimension}

\begin{abstract}

Let $R=\oplus_{i\in \N_0}R_n$ be a standard graded ring, $R_+ :=\oplus_{i\in \N}R_n$ be the irrelevant ideal of $R$ and $\fa_0$ be an ideal of $R_0$. In this paper, as a generalization of the concept of Castelnouvo-Mumford regularity $\reg(M)$ of a finitely generated graded $R$-module $M$, we define the regularity of $M$ with respect to $\fa_0+ R_+$, say $\reg_{\fa_0+ R_+}(M)$. And we study some relations of this new invariant with the classic one. To this end,
 we need to consider the cohomological dimension of some finitely generated $R_0$-modules. Also, we will express $\reg_{\fa_0+ R_+}(M)$ in terms of some invariants of the minimal graded  free resolution of $M$ and see that in a special case this invariant is independent of the choice of $\fa_0$.

\end{abstract}
\maketitle

\section{Introduction}
Let $R=\oplus_{i\in \N_0}R_n$ be a standard graded algebra; so that $R_0$ is a Noetherian ring and $R$ is a $\N_0$-graded $R_0$-algebra which can be generated by finitely many elements of $R_1$. Also, assume that $M$ is a finitely generated graded $R$-module. Then, it is well-known that for all $i\in \N_0$, $H^i_{R_+}(M)$ carries a natural grading  and that $H^i_{R_+}(M)_n =0$ for all $n\gg 0$ (see \cite[chap. 13]{BSH}).
An important concept concerning this fact, called Castelnouvo-Mumford regularity of $M$, is defined to be
$$\reg_{R_+}(M):= \sup\{\en(H^i_{R_+}(M))+ i| i\in \N_0\},$$
where for a graded $R$-module $X=\oplus_{i\in \Z}X_n$, $\en(X):= \sup\{n\in \Z| X_n\neq 0\}$ with $\sup\{\emptyset\}:= -\infty$.

In addition, according to \cite{JZ} (\cite {hyry}, \cite{sh} or \cite{trung}), if $(R_0, \m_0)$ is local and $\fa_0$ is an ideal of $R_0$ then $H^i_{\fa_0+ R_+}(M)_n= 0$ for all $i\in \N_0$ and all $n\gg 0$. There, they generalized the $a^*$-invariants of $M$, defined by $a^*(M):= \sup\{\en(H^i_{R_+}(M))| i\in \N_0\}$, to the $a^*$-invariants of $M$ with respect to $\fa_0+ R_+$
by $a^*_{\fa_0+ R_+}(M):= \sup\{\en(H^i_{\fa_0+ R_+}(M))| i\in \N_0\}(<\infty)$ and showed that these two invariants are actually equal, i.e.
$$a^*_{\fa_0+ R_+}(M)= a^*(M).$$

As a generalization of the concept $\reg_{R_+}(M)$ we define the regularity of $M$ with respect to the ideal $\fa_0+ R_+$ at and above level $k (\in \N_0)$  by
$$\reg_{\fa_0+ R_+}^k(M):= \sup\{\en(H^i_{\fa_0+ R_+}(M))+ i| i \geq k\}$$
(see \cite{h} and \cite{JK}) and we set $\reg_{\fa_0+ R_+}(M):= \reg_{\fa_0+ R_+}^0(M)$.

Now, in view of the above equality of $a^*$-invariants, it is natural to ask:
\begin{ques}\label{q}
Does the equality $\reg_{\fa_0+ R_+}(M)= \reg_{R_+}(M)$ hold?
\end{ques}
In Section 2 of this paper we consider this question and show that it is not the case in general. However, there are some relations between them in terms of some cohomological dimensions (see \ref{reg.cd} and \ref{r cm}). Also, using the graded minimal free resolution of $M$, we show that, in a special case, $\reg_{\fa_0+ R_+}(M)$ is independent of the choice of $\fa_0$ (see \ref{res}).

In Section 3, in view of the results in Section 2 concerning relations between cohomological dimensions and regularity, we study some cohomological dimension formulas.

Throughout the paper, unless otherwise stated, $R=\oplus_{i\in \N_0}R_n$ will denote a standard graded ring, $\fa_0$ an ideal of $R_0$ and $M$ a finitely generated graded $R$-module.
\section{ Castelnuovo-Mumford regularity}

The following example answer the question \ref{q} in a negative way.

\begin{exam}
Let $(R_0, \m_0)$ be local with $d:= \dim(R_0)> 0$ and consider $R_0$ as a graded $R$-module which is concentrated in degree $0$.Then for all $i\in \N_0$, in view of \cite[13.1.10]{BSH}, we have
\[ H^i_{\m_0+ R_+}(R_0)_n\cong H^i_{\m_0R}(R_0)_n\cong \left\lbrace
  \begin{array}{c l}
   H^i_{\m_0}(R_0), & n=0;\\
  0, & n\neq 0.
  \end{array}
\right. \]

Which implies that $\reg_{\m_0+ R_+}(R_0)= d$. Also,
\[ H^i_{R_+}(R_0)_n= \left\lbrace
  \begin{array}{c l}
   R_0, &i= 0= n;\\
 0, & otherwise.
  \end{array}
\right. \]
Therefore, $\reg_{R_+}(R_0)=0 < d= \reg_{\m_0+ R_+}(R_0).$
\end{exam}

\begin{defen}
The arithmetic rank of an ideal $\fa$ of $R$, denoted by $\ara(\fa)$, is defined to be
$min \{n\in \N_0| \exists x_1, \cdots, x_n\in R \ \text{such that} \ \rad(\fa)=\rad((x_1, \cdots, x_n))\}$.
\end{defen}

Although, according to the above example, $\reg_{\fa_0+ R_+}(M)\neq \reg_{R_+}(M)$ but we have the following relation between them.

\begin{prop}\label{reg}
$\reg_{\fa_0+ R_+}(M)\leq \reg_{R_+}(M)+ \ara(\fa_0)$.
\end{prop}

\begin{proof}
Assume that $\ara(\fa_0)= t$ and $x_1, \cdots, x_t\in R_0$ such that $\rad(\fa_0)=\rad((x_1, \cdots, x_t))$. We use induction on $t$.
Let $t=1$ and $x:=x_1$. Then, in view of \cite[5.1.22]{BSH}, for all $n\in \Z$ and all $i\in \N_0$ there is an exact sequence
$$(H^{i-1}_{R_+}(M)_x)_{n-1-(i-1)}\longrightarrow H^i_{x+ R_+}(M)_{n-i}\longrightarrow H^i_{R_+}(M)_{n-i}$$
of $R_0$-modules.
For all  $n> \reg(M)+ 1$ and all $i\in \N_0$ we have $H^i_{R_+}(M)_{n-i}= 0= (H^{i-1}_{R_+}(M)_x)_{n-1-(i-1)}$.
So, the above exact sequence implies that for all $n> \reg(M)+ 1$, $H^i_{x+ R_+}(M)_{n-i}=0$ for all $i\in \N_0$. Therefore $\reg_{x+ R_+}(M)\leq \reg_{R_+}(M)+ 1$.
Which proves the claim in the case $t =1$.

Now, the same argument as used above completes the induction.
\end{proof}

\begin{defen rem}
The cohomological dimension of an $R$-module $X$ with respect to an ideal $\fa$ of $R$, which is an important invariant of $X$, denoted by $\cd_{\fa}(X)$ and is defined to be the last integer $i$ for which $H^i_{\fa}(X)\neq 0$.

In view of \cite[3.3.1]{BSH}, one has $\cd_{\fa}(X)\leq \ara(\fa)$.
\end{defen rem}

In the rest of this section we consider two special cases:

 \noindent\textbf{$\bullet$ The case where $R= R_0[x_1, \cdots, x_t]$ and $M= M_0[x_1, \cdots, x_t]$:}

 In this subsection we assume that $R= R_0[x_1, \cdots, x_t]$ is the standard graded polynomial ring over a Noetherian ring $R_0$, $M_0$ is an $R_0$-module and $M:= M_0[x_1, \cdots, x_t]$ is a polynomial module over $R$ which is graded in the usual way.

The following theorem, in conjunction with its two next results, shows that in order to study the regularity of some graded $R$-modules with respect to an ideal $\fa_0+ R_+$ it is worth to study the cohomological dimension
of some of its graded components with respect to $\fa_0$.

\begin{thm}\label{reg.cd}

$\reg_{\fa_0+ R_+}(M)= \cd_{\fa_0}(M_0).$
\end{thm}
\begin{proof}
Since $x_1, \cdots, x_t$ is a regular sequence on $M$ so, in view of \cite[3.4]{NS} and \cite[2.1.9 and 13.1.10]{BSH}, for all $n\in \Z$ there are the following isomorphisms of $R_0$-modules
\[H^i_{\fa_0+ R_+}(M)_n\cong \left\lbrace
  \begin{array}{c l}
   H^{i- t}_{\fa_0}(H^t_{R_+}(M)_n), & i\geq t;\\
   0, & i<t.
  \end{array}
\right.\]
Now, for all $i\geq t$, according to \cite[12.4.1]{BSH}, one has the following isomorphisms of graded $R$-modules
$$H^{i- t}_{\fa_0R}(H^t_{R_+}(M))= H^{i- t}_{\fa_0R}(H^t_{R_+}(M_0[x_1, \cdots, x_t]))\cong H^{i- t}_{\fa_0R}(M_0[x_1^-, \cdots, x_t^-])\cong H^{i- t}_{\fa_0}(M_0)[x_1^-, \cdots, x_t^-],$$
where for an $R_0$-module $X_0$, $X_0[x_1^-, \cdots, x_t^-]$ denotes the module of inverse polynomials in $x_1, \cdots, x_t$ over $X_0$; so that it has a graded $R$-module structure such that, for $(i_1, \cdots, i_t)\in (-\N)^t$ and $1\leq s\leq t$,
\[x_s(x_1^{i_1} \cdots x_t^{i_t})= \left\lbrace
  \begin{array}{c l}
  x_1^{i_1} \cdots x_{s- 1}^{i_{s- 1}} x_s^{i_s+ 1}x_{s+ 1}^{i_{s+ 1}} \cdots  x_t^{i_t}& i_s< -1;\\
   0, & i_s= -1.
  \end{array}
\right.\]
Therefore, for all $i\geq t$, if $H^{i- t}_{\fa_0}(M_0)\neq 0$ then $\en(H^i_{\fa_0+ R_+}(M))= -t$. Which implies that $\reg_{\fa_0+ R_+}(M)= \cd_{\fa_0}(M_0)$.
\end{proof}

The following corollary, which is immediate by the above theorem and \cite[3.3.1]{BSH}, consider a case in which one has equality in \ref{reg}.

\begin{cor}
Assume that $a_0$ is generated by a regular sequence of length $n$ on $M_0$. Then, $\reg_{\fa_0+ R_+}(M)=
 \ara(\fa_0)= n$.
\end{cor}

\begin{cor}
Assume that $(R_0, \m_0)$ is local and $M_0$ is Cohen-Macaulay. Then
$$\reg_{\fa_0+ R_+}(M)= \dim(M_0)- \inf\{i\in \N_0| \lim_{\stackrel{\longleftarrow}{n}}H^i_{\m_0}(\frac{M_0}{\fa_0^n  M_0})\neq 0\}.$$
\end{cor}

\begin{proof}
The result follows from Theorem \ref{reg.cd} and \cite[4.1]{DA}.
\end{proof}

\begin{cor}
Assume that $R_0$ contains a field and $\fa_0$ is an ideal of $R_0$ which can be generated by monomials in an $R_0$-sequence. Then
$$\reg_{\fa_0+ R_+}(R)= \pd_{R_0}(R_0/ \fa_0)= \depth(R_0)- \depth(R_0/ \fa_0).$$
\end{cor}

\begin{proof}
The claim can be proved using Theorem \ref{reg.cd} and \cite[4.3]{ST}.
\end{proof}

\begin{cor}
Assume that the base ring $(R_0, \m_0)$ is local. Then
$$\reg_{\fa_0+ R_+}(M_0)\geq\{\dim(M_0/\fb_0 M_0)| \fb_0 \ {\text is \ an \ ideal \ of \ }R_0 {\text \  and \ } \rad(\fa_0+ \fb_0)= \m_0\}.$$
\end{cor}

\begin{proof}
Let $\fb_0$ be an ideal of $R_0$ such that $\rad(\fa_0+ \fb_0)= \m_0$. Then, using \cite[6.1.4]{BSH} and \cite[2.2]{dnt}, one has
$$\cd_{\fa_0}(M_0)\geq \cd_{\fa_0}(M_0/\fb_0 M_0)= \cd_{\m_0}(M_0/\fb_0 M_0)= \dim(M/\fb_0 M_0).$$
Now, the result follows from Theorem \ref{reg.cd}.
\end{proof}

\noindent\textbf{$\bullet$ The case of only one non vanishing local cohomology module:}

\begin{defin}
Following \cite{JR}, an $R$-module $N$ is called relative Cohen-Macaulay of rank $g$ with respect to an ideal $\fb$ of $R$ if
$g= \grade_{\fb}(N)= \cd_{\fb}(N)$, i.e. $H^i_{\fb}(N)\neq 0$ if and only if $i= g$.
\end{defin}

The following proposition studies $\reg_{\fa_0+ R_+}(M)$ when $M$ is relative Cohen-Macaulay with respect to $R_+$.

\begin{prop}\label{r cm}
Let $M$ be relative Cohen-Macaulay with respect to $R_+$ of rank $g$. Then $$\reg_{\fa_0+ R_+}(M)= \sup\{\cd_{\fa_0}(H^g_{R_+}(M)_n)+n| n\in \Z\}+ g.$$
\end{prop}

\begin{proof}
Consider the Grothendieck spectral sequence $E_2^{i, j}= H^{i}_{\fa_0R}(H^{j}_{R_+}(M))\underset{i}{\Rightarrow}H^{i+ j}_{\fa_0+ R_+}(M)$. Since $M$ is relative Cohen-Macaulay with respect to $R_+$ of rank $g$, so $H^j_{R_+}(M)= 0$, for all $j\neq g$. Therefore, in view of the concept of convergence of spectral sequences, we have $H^{i}_{\fa_0+ R_+}(M)= 0$ for all $i< g$ and there are homogenous isomorphisms
$$H^{i}_{\fa_0R}(H^{g}_{R_+}(M))= E_2^{i, g}\cong E_{\infty}^{i, g}\cong H^{i+ g}_{\fa_0+ R_+}(M)$$
of graded $R$-modules for all $i\in \N_0$.

So, using \cite[13.1.10]{BSH}, we have
\begin{eqnarray*}
\reg_{\fa_0+ R_+}(M) &= & \sup \{ \en(H^{i}_{\fa_0+ R_+}(M))+ i | i\in \N_0\}\\
        &= & \sup \{ \en(H^{i+ g}_{\fa_0+ R_+}(M))+ i+ g | i\in \N_0\}\\
        &= & \sup \{ \en(H^{i}_{\fa_0R}(H^{g}_{R_+}(M)))+ i+ g | i\in \N_0\}\\
        &= & \sup \{ \sup \{n\in \Z | H^{i}_{\fa_0}(H^{g}_{R_+}(M)_n)\neq 0\}+ i | i\in \N_0\}+ g.
\end{eqnarray*}
Now, it is straightforward to see that
$$\sup \{ \sup \{n\in \Z | H^{i}_{\fa_0}(H^{g}_{R_+}(M)_n)\neq 0\}+ i | i\in \N_0\}= \sup\{\cd_{\fa_0}(H^g_{R_+}(M)_n)+n| n\in \Z\},$$
which proves the claim.
\end{proof}

In Proposition \ref{reg} we find a general relation between the invariants $\reg_{R_+}(M)$ and $\reg_{\fa_0+ R_+}(M)$. The following corollary provides another one in a special case.
A similar bound was obtained in \cite[3.4]{h}, in the case where $R$ is Cohen-Macaulay.

\begin{cor}
Let the situation be as above. Also, assume that the base ring $(R_0, \m_0)$ is local. Then

(i) $\reg_{\m_0+ R_+}(M)= \sup\{\dim_{R_0}(H^d_{R_+}(M)_n)+n| n\in \Z\}+ d\leq \dim(R_0)+d$,
where $d:= \dim(M/\m_0 M)$.

(ii) $\reg_{R_+}(M)\leq \reg_{\fa_0+ R_+}(M)$ for any ideal $\fa_0$ of $R
_0$.
\end{cor}

\begin{proof}
Considering Proposition \ref{r cm}, (i) follows from \cite[2.3]{B} and (ii) follows from the concept of $\reg_{R_+}(M)$.
\end{proof}

In the following theorem we are going to express $\reg_{\fa_0+ R_+}(M)$ in terms of some invariants of the minimal graded free resolution of $M$. As we will see, in a special case, $\reg_{\fa_0+ R_+}(M)$ is independent of the choice of $\fa_0$!

\begin{thm}\label{res}
Let $(R_0, \m_0)$ be local and $M$ and $R$ be relative Cohen-Macaulay with respect to $R_+$ of the same rank with $p:=\pd_R(M)< \infty$. Also, assume that $\reg_{\fa_0+ R_+}(R)= 0$ and let
$$0 \longrightarrow \oplus_{i=1}^{n_p}R(a_{p}^{i})\stackrel{d_p}{\longrightarrow}\oplus_{i=1}^{n_{p-1}}R(a_{p-1}^{i})\stackrel{d_{p-1}}{\longrightarrow}\cdots \stackrel{d_1}{\longrightarrow} \oplus_{i=1}^{n_0}R(a_{0}^{i})\stackrel{d_0}{\longrightarrow} M\longrightarrow 0$$
be the minimal graded free resolution of $M$. Then
$$\reg_{\fa_0+ R_+}(M)= \max_{i= 0}^{p}\big(-\min_{j=1}^{n_i}a_{i}^{j}-i\big).$$
\end{thm}
\begin{proof}
We use induction on $p$ to prove the claim. If $p= 0$ then $M\cong \oplus_{i=1}^{n_0}R(a_{0}^{i})$. So,
$$\reg_{\fa_0+ R_+}(M)= \max_{i=1}^{n_0}(\reg_{\fa_0+ R_+}(R(a_{0}^{i})))= \reg_{\fa_0+ R_+}(R)- \min_{i=1}^{n_0}a_{0}^{i}= - \min_{i=1}^{n_0}a_{0}^{i}.$$
Now, let $p> 0$. Then, using \cite[15.2.15(i)]{BSH}, one can see that

\begin{eqnarray*}
\reg_{\fa_0+ R_+}(\ker(d_0))&\leq& \max\{\reg_{\fa_0+ R_+}(\oplus_{i=1}^{n_0}R(a_{0}^{i})), \reg_{\fa_0+ R_+}(M)+ 1\}\\
&=& \max\{\reg_{\fa_0+ R_+}(R)- \min_{i=1}^{n_0}(a_{0}^{i}), \reg_{\fa_0+ R_+}(M)+ 1\}\\
&=&\max\{- \min_{i=1}^{n_0}(a_{0}^{i}), \reg_{\fa_0+ R_+}(M)+ 1\} \\
&\leq& \reg_{\fa_0+ R_+}(M)+ 1, \ \ \ \ \ \ \ \ \ \ \ \ \ \ \ \ \ \ \ \ \ \ \ \ \ \ \ \ \ \ \ \ \ \ \ \ \ \ \ \ \ \ \ \ \ \ \ \ \ \ \ \ \ \ \ \ \ \ \ (*)
\end{eqnarray*}
where the last inequality follows from the second part of the above corollary and \cite[15.3.1]{BSH}.
On the other hand, in view of facts that $\ker(d_0)\neq 0$ and $\cd_{R_+}(\ker(d_0))\leq \cd_{R_+}(R)$ and the exact sequence $0\rightarrow \ker(d_0)\rightarrow \oplus_{i=1}^{n_0}R(a_{0}^{i})\stackrel{d_0}{\rightarrow} M\rightarrow 0$, $\ker(d_0)$ is relative Cohen-Macaulay with respect to $R_+$ of the same rank as $R$. So, using inductive hypothesis, we have
$$\reg_{\fa_0+ R_+}(\ker(d_0))= \max_{i= 1}^{p}\big(-\min_{j=1}^{n_i}a_{i}^{j}- (i- 1)\big),$$
which in conjunction with $(*)$ implies that
\begin{eqnarray*}
\reg_{\fa_0+ R_+}(M)&\geq& \max\{\max_{i= 1}^{p}\big(-\min_{j=1}^{n_i}a_{i}^{j}- i\big), -\min_{i=1}^{n_0}a_{0}^{i}\}\\
&=& \max_{i= 0}^{p}\big(-\min_{j=1}^{n_i}a_{i}^{j}- i\big).
\end{eqnarray*}
Also, by \cite[15.2.15(iv)]{BSH},
\begin{eqnarray*}
\reg_{\fa_0+ R_+}(M)&\leq& \max\{\reg_{\fa_0+ R_+}(\oplus_{i=1}^{n_0}R(a_{0}^{i})), \reg^1_{\fa_0+ R_+}(\ker(d_0))- 1\}\\
&\leq& \max\{\reg_{\fa_0+ R_+}(R)- \min_{i=1}^{n_0}(a_{0}^{i}), \reg_{\fa_0+ R_+}(\ker(d_0))- 1\}\\
&=&\max\{- \min_{i=1}^{n_0}(a_{0}^{i}), \max_{i= 1}^{p}\big(-\min_{j=1}^{n_i}a_{i}^{j}- (i- 1)\big)- 1\} \\
&\leq& \max_{i= 0}^{p}\big(-\min_{j=1}^{n_i}a_{i}^{j}- i\big).
\end{eqnarray*}
Now, the result follows by induction.
\end{proof}

\section{ Cohomological Dimension}

As we have seen in Section 2, in order to calculate the regularity of a finitely generated graded module over some standard graded ring with respect to an ideal containing the irrelevant ideal, we need to calculate some cohomological dimensions. In this section, we are going to study this invariant in some cases.

\begin{defin}
Let $(R, \m)$ be a local ring of positive characteristic $p$. Also, assume that for all $s\in \N$,  $F^s$ denotes the $R$-homomorphism induced by the $s$-th power of frobenius homomorphism $f^s: R\rightarrow R \ \ (r\rightarrow r^{p^s})$ on $H^i_{\m}(R)$ (so, it is a homomorphism of abelian groups such that $f(rx)= r^{p^s}x$ for all $r\in R$ and all $x\in H^i_{\m}(R)$). Then, following \cite{lyu}, the F-depth of $R$ is defined to be the smallest $i$ such that $F^s$ does not send $H^i_{\m}(R)$ to zero for any $s$. We denote this number by $\Fdepth(R)$.
\end{defin}

Let $(R, \m)$ be a regular local ring of positive characteristic and $\fa$ be an ideal of $R$. Then, in view of \cite[Theorem 4.3]{lyu}, we have $\cd_{\fa}(R)=\dim(R)- \Fdepth(R/ \fa)$.
So, the cohomological dimension may differ depending on the characteristic of the ring, for an example see Section 5 of \cite{lyu}.
In the next theorem we are going to generalize this result to some special ideals in a Cohen-Macaulay local ring of positive characteristic.

\begin{lem}
Let $\varphi: (\acute{R}, \acute{\m})\rightarrow (R, \m)$ be a flat homomorphism of local rings of positive characteristic for which $\phi(\acute{\m})= \m$. Then  $\Fdepth(\acute{R})= \Fdepth(R)$.
\end{lem}

\begin{proof}
In view of the flat base change theorem \cite[4.3.2]{BSH}, we have the following commutative diagram where the vertical homomorphisms are isomorphisms:
\[
\xymatrix{\ar @{} [dr] |{}
H^i_{\acute{\m}}(\acute{R}) \ar[d]_{-\otimes_{\acute{R}} R} \ar[r]^{F^s} & H^i_{\acute{\m}}(\acute{R}) \ar[d]_{-\otimes_{\acute{R}} R} \\
H^i_{\m}(R) \ar[r]^{F^s} & H^i_{\m}(R) }
\]

Now, the result follows using the concept of F-depth.
\end{proof}

\begin{rem}(An analogue of Noether normalization)
Let $(R, \m)$ be a Cohen-Macaulay local ring of positive characteristic $p$ and $\hat{R}$ denote the $\m$-adic completion of $R$. Then, in view of \cite[A 20]{BH}, $\hat{R}$ contains a coefficient field $k$. Also,using \cite[A 22]{BH}, for any system of parameters $y_1, \cdots, y_d$ of $R$, $\hat{R}$ is a finite module over the regular local ring $k[[y_1, \cdots, y_d]]$.
\end{rem}
\begin{thm}\label{c.m}
Let $(R, \m)$ be a Cohen-Macaulay local ring of positive characteristic $p$ and $\fa$ be an ideal of $R$ which can be generated by elements in $k[y_1, \cdots, y_d]$ for a system of parameters $y_1, \cdots, y_d$ of $R$ and a coefficient field $k\subseteq \hat{R}$.
Then $\cd_{\fa}(R)=\dim(R)- \Fdepth(R/ \fa)$.
\end{thm}
\begin{proof}

Since $\cd_{\fa}(R)= \cd_{\fa \hat{R}}(\hat{R})$, in view of the previous lemma and replacing $R$ with $\hat{R}$, we may assume, in addition, that $R$ is a complete local ring. Now, using the above remark, $R$ is a finite module over the regular local ring $\acute{R}:= k[[y_1, \cdots, y_d]]$ under the natural homomorphism $\varphi: \acute{R}\rightarrow R$.
In view of these facts and \cite[1.2.26(b)]{BH} we have
$$\infty> \pd_{\acute{R}}(R)=\dim(\acute{R})- \depth_{\acute{R}}(R)= \dim(R)- \depth_{R}(R)= 0.$$
Therefore, $R$ is a free $\acute{R}$-module. Which implies that $\varphi$ is a faithfully flat homomorphism. Also, by the assumption on $\fa$ we have $(\fa \cap \acute{R})R= \fa$. So, the natural homomorphism $\frac{\acute{R}}{\fa \cap \acute{R}}\rightarrow \frac{R}{\fa}$, induced by $\varphi$, is faithfully flat. Now, using the previous lemma, one has
$$\Fdepth(\acute{R}/ \fa\cap \acute{R})= \Fdepth(R/\fa).$$
Again, using the faithful flatness of $\varphi$, and in view of \cite[Theorem 4.3]{lyu}, we have
$$\dim(R)- \Fdepth(R/ \fa)= \dim(\acute{R})- \Fdepth(\acute{R}/ \fa\cap \acute{R})= \cd_{\fa\cap \acute{R}}(\acute{R})= \cd_{(\fa\cap \acute{R})R}(\acute{R}\otimes _{\acute{R}} R)= \cd_{\fa}(R).$$

\end{proof}

The following corollary is a consequence of \cite[Theorem 4.3]{lyu} and theorems \ref{c.m} and \ref{reg.cd}.

\begin{cor}
Let $S= R [X_1,\cdots, X_n]$ be the standard graded polynomial ring over $R$. Then
$\reg_{\fa+ S_+}(S)= \dim(R)- \Fdepth(R/ \fa)$ in each of the following cases

(i) $R$ is a regular ring of positive characteristic  and $\fa$ is an ideal of $R$,

(ii)  $R$ and $\fa$ be as in Theorem \ref{c.m}.

\end{cor}

\textbf{Acknowledgments}\\

 The author would like to thank the referee for his/her valuable comments.

\end{document}